\documentclass{amsart}

\newtheorem{Theorem}{Theorem}[section]
\newtheorem{Lemma}[Theorem]{Lemma}

\begin{document}

\date{December 27, 2010}

\title{Symmetry groups of boolean functions}
\author{Mariusz Grech and Andrzej Kisielewicz}

\address{University of Wroc\l aw\\
Institute of Mathematics\\
pl. Grunwaldzki 2, 50-384 Wroc\l aw, Poland}
\email{mariusz.grech@math.uni.wroc.pl}
\email{andrzej.kisielewicz@math.uni.wroc.pl}
\thanks{Supported in part by Polish MNiSZW grant N N201 543038.}

\begin{abstract}
We prove that every abelian permutation group, but known exceptions, is the symmetry group of a boolean function. This solves the problem posed in the book by Clote and Kranakis. In fact, our result is proved for a larger class of groups, namely, for all groups contained in direct sums of regular groups.
\end{abstract}

\maketitle

We investigate the problem of representability of permutation groups by the symmetry groups of boolean functions. For a permutation group $G\leq S_n$ we consider its natural action on the set $\{0,1\}^n$ given by
$$
x^\sigma = (x_{\sigma(1)}, \ldots, x_{\sigma(n)}),
$$
where $\sigma\in S_n$ and $x = (x_1,\ldots,x_n) \in \{0,1\}^n$.

The \emph{symmetry} (or \emph{invariance}, or \emph{automorpism}) \emph{group} of a boolean function $f : \{0,1\}^n \to \{0,1\}$ is the group $G(f) \leq S_n$ defined as follows
$$
G(f) = \{\sigma : f(x^\sigma) = f(x)\}.
$$
A permutation group $G \leq S_n$ is representable as the symmetry group of a boolean function, or in short, \emph{representable}, if there exists a boolean function $f$ such that $G=G(f)$.

Not all permutation groups are representable. For example alternating groups $A_n$ are not. The problem of representability by the invariance groups of boolean functions was first considered by Clote and Kranakis \cite{CK1} in connection with parallel complexity of formal languages and the upper bounds for complexity of boolean circuits (see \cite{CK2}, chapter 3). They established the representability conditions for cyclic groups (generated by a single permutation) and for maximal subgroups of $S_n$. In the book \cite{CK2} they asked about similar results for abelian groups (Exercise 3.11.15 (Open Problem), p.~197). In \cite{Gre}, Grech has proved that all regular permutation groups are representable but a few know exceptions. In this paper we consider permutation groups which are subgroups of direct sums of regular groups. The latter contains, in particular, abelian groups and generalized dicyclic groups.

The topic is closely connected with two areas of current research in algebra and discrete mathematics: defining permutation groups by relations (see \cite{Sim,Ser,Wie}) and  automorphism groups of graphs and other discrete structures (see \cite{Bab,God,Kis2}). 

\section{Preliminaries}

We consider finite permutation groups up to permutation isomorphism. Thus, generally, we assume that a permutation group $G$ is a subgruop of the symmetric group $S_n$ of the set $X= \{1,2,\ldots, n\}$. The main construction considered is that of the direct sum of permutation groups. Given two groups $G\leq S_n$ and $H\leq S_m$, the \emph{direct sum} $G\oplus H$ is the permutation group on $\{1,2,\ldots, n+m\}$ defined as the set of permutations $\pi =(\sigma,\tau)$ such that
$$
\pi(i) =
\left\{\begin{array}{ll}
\sigma(i), & \mbox{if }  i\leq n\\
n+\tau(i), & \mbox{otherwise.}
\end{array}\right.
$$
Thus, in $G\oplus H$, permutations of $G$ and $H$ act independently in a natural way on a disjoint union of the base sets of the summands. This construction is often called the ``direct product'' and denoted with $\times$, but we found that in view of other constructions it is more appropriate to use the term ``sum'' (as some authors do; see e.g.,  \cite{KP}).

We adopt general terminology of permutation groups as given, for example, in \cite{DD}. In addition, we introduce the notion of the subdirect sum following the notion of ``intransitive product'' formulated in \cite{Kis}.

Let $H_1\lhd\; G_1\leq S_n$ and $H_2\lhd\; G_2 \leq S_m$ be permutation groups such that $H_1$ and $H_2$ are normal subgroups of $G_1$ and $G_2$, respectively. Suppose, in addition, that factor groups $G_1/H_1$ and $G_2/H_2$ are (abstractly) isomorphic and $\phi :  G_1/H_1 \to G_2/H_2$ is the isomorphism mapping. Then, by
$$
G_1/H_1 \oplus_\phi G_2/H_2
$$
we denote the subgroup of $G_1 \oplus G_2$ consisting of all permutations $(\sigma,\tau)$ such that $\phi(\sigma H_1) = \tau H_2$. In general, such a group is called a \emph{subdirect sum} of $G_1$ and $G_2$, and denoted by $G_1 \oplus_\phi G_2$ or $G_1 \oplus G_2 (\phi)$ (in such a case the normal divisors $H_1$ and $H_2$ are assumed to be hidden in the full description of the isomorphism $\phi$).  The term ``subdirect'' comes from universal algebra and is choosen to point out that for each permutation $\sigma\in G_1$ there is a permutation $\tau\in G_2$ such that $(\sigma,\tau)\in G_1 \oplus_\phi G_2$, and conversely---for each permutation $\tau\in G_2$ there is a permutation $\sigma\in G_1$ such that $(\sigma,\tau)\in G_1 \oplus_\phi G_2$. In other words, $G_1\oplus G_2$ is the least direct sum containing $G_1 \oplus_\phi G_2$.

We generalize this into a larger number of summands by recursion.
$$
G_1\oplus G_2\oplus  \ldots \oplus G_k (\Phi_k) = G_1 \oplus_\phi (G_2\oplus  \ldots \oplus G_k (\Phi_{k-1})),
$$
where $\Phi_k = (\phi, \Phi_{k-1})$ is a sequence of isomorphisms describing the permutations in the sum. For each permutation $\sigma\in G_1$ there are permutations $\tau_2,\ldots,\tau_k$ such that $(\sigma,\tau_2,\ldots,\tau_k)\in D$. The same holds for any other summand $G_j$. As a matter of fact, under suitable convention, the operation of subdirect sum may be treated as associative and commutative. A summand $G_j$ is called \emph{independent} if $D$ can be represented as $D=G_j\oplus C$; it means that for all permutations $\sigma\in G_j$ and $\tau\in C$, $(\sigma,\tau)\in D$. Otherwise, $G_j$ is \emph{dependent on} (one of) other summands.

We will consider also the following generalization of boolean functions. By a $k$-\emph{valued boolean function} we mean a map of the form $f : \{0,1\}^n \to \{0,1,\ldots,k-1\}$. The definition of the symmetry group $S(f)$ is the same as in the 2-valued case. By $BGR(k)$ we denote the set of all permutation groups that are symmetry groups of $k$-valued boolean functions. Such functions are called $k$-\emph{representable} (thus ,,representable'' means ,,$2$-representable'').
In \cite{CK1}, Clote and Kranakis has formulated a result implying that $BGR(k) = BGR(2)$ for any $k\geq 2$. Yet, the proof of this result turned out to be false. Kisielewicz \cite{Kis} has observed that the group $K_4 \leq S_4$ generated by two permutations $(1,2)(3,4)$ and $(1,3)(2,4)$ (isomorphic abstractly to the Klein four-group) is in $BGR(3)$, but not in $BGR(2)$. No other counterexample of this kind has been found so far. On the other hand, there are some results confirming the conjecture by Clote and Krankis. Since we apply these result in the sequel, we recall them now in a precise form.

By $C_i \leq S_i$ we denote the permutation group contained generated by the cycle $\sigma= (1,2,\ldots, n)$. It is not difficult to check (see \cite{CK1,Kis} that $C_i\in BGR(2)$ for $i\neq 3,4,5$, while $C_3,C_4,C_5 \notin BGR(k)$ for any $k\geq 2$.

\begin{Theorem}[Clote, Kranakis \cite{CK1}]\label{cycCK} If $G\leq S_n$ is a permutation group generated by a single permutation $\sigma$ (cyclic as an abstract group), then
either $G\in BGR(2)$ or $G\notin BGR(k)$ for any $k\geq 2$. Moreover, if $\sigma$ is a product of $k$ disjoint cycles of length $l_1,l_2,\ldots,l_k \geq 2$, respectively, then $G\in BGR(2)$ if and only if for all $s=3,4,5$ and $i\leq k$ the equality $l_i = s$ implies that there is $j\neq i$ such that $\gcd(l_i,l_j) \neq 1$.
\end{Theorem}

In particular,
if $\sigma = (1,2,3)(4,5,6)$, then $G\in BGR(2)$. Note that this cyclic group may be presented as an subdirect sum $G = C_3/1 \oplus_\phi C_3/1$. (By $1$ we denote both the identity permutation and the trivial subgroup consisting of it). The permutation groups of this form, consisting of two copies of the same group $G$ acting in parallel manner, are called \emph{parallel sums} and we shall denote them $G^{(2)}$. On the other hand, the direct sum $C_3\oplus C_3$, as we shall see in Section~2, is not representable at all. A similar situation is for $C_4$ and $C_5$. The parallel sums $C_4^{(2)}$ and $C_5^{(2)}$ are in $BGR(2)$, while the direct sums $C_4\oplus C_4$  and $C_5\oplus C_5$ are not in $BGR(k)$ for any $k\geq 2$. 
In case of $C_4$, we have also the subdirect sum $C_2 \oplus_\phi C_4/C_2^{(2)}$, which may be seen to be a cyclic group and according to the theorem above is representable (in this notation we have omitted the normal factor $1$ in the first summand; note also that the this notation determines uniquely the permutation group up to permutation isoomorphism).

Recently, M. Grech \cite{Gre} characterized representability for regular and semiregular permutation groups proving, in particular, that a regular permutation group is representable if and only if it is different from $C_3,C_4$ and $C_5$.

A closely connected topic is research on defining permutation groups by relations, and especially that concerning unordered relations (see \cite{Sim,Ser} and the references given therein). An \emph{unordered relation} $\mathbf R$ is simply a set of subsets a given set $X$. We consider the natural action of the symmetric group $S(X)$ on the subsets of $X$, and by $S^\sigma$ we denote the image of the set $S\subseteq X$ under this action. Then, by $G(\mathbf R)$ we denote the subgroup of $S(X)$ consisting of those permutations $\sigma$ which leave $\mathbf R$ invariant, that is, $S^\sigma \in \mathbf R$ for all $S\in \mathbf R$.

There is a natural one-to-one correspondence between unordered relations and boolean functions. Given a subset $S$ of an $n$-element set $X$ by $x_S$ we denote the $n$-tuple corresponding to the characteristic function of $S$ (and a fixed linear ordering of elements of $X$). Then the function given by
$$
f(x_S) =
\left\{\begin{array}{ll}
1, & \mbox{if }  S\in\mathbf R\\
0, & \mbox{otherwise,}
\end{array}\right.
$$
is a boolean function on $\{0,1\}^n$ (determining the relation $\mathbf R$).
In particular, a group $G$ is representable if and only if $G = S(\mathbf R)$ for some unordered relation $\mathbf R$.

In \cite{Sim}, Dalla Volta and Siemons have used results on regular sets in permutation groups to obtain further results on representability. For a permutation group $G$ on a set $X$ a set $S \subseteq X$ is called \emph{regular in} $G$ if for all $\sigma\in G$, $S^\sigma = S$ implies $\sigma = 1$. In \cite{Sim} it is proved that if a permutation group $H = S(\mathbf R)$ and $H$ has a regular set $S$ such that there is no set of cardinality $|S|$ in $\mathbf R$, then \emph{every} subgroup of $H$ is representable. In particular, it is proved that if $G$ is a subgroup of a primitive group other than $A_n$ and $S_n$, then with a few possible exceptions, $G$ is representable.

Earlier, as we have already mentioned, the representability for maximal subgroups of $S_n$ has been characterized (\cite{CK1}).
Summarizing, the situation is such, that apart from $K_4$, we know only permutation groups representable by boolean functions (those in $BGR(2))$, and not representable at all (those not in $BGR(k)$ for any $k\geq 2$). The two main open problems in the area are 1) to settle whether the conjecture by Clote and Kranakis that $BGR(k) = BGR(2)$ is true in principle (i.e., with only few exceptions), and 2) to give a characterization of the class $BGR(2)$ of permutation groups representable as the invariance groups of boolean functions.

In this paper we combine approaches in \cite{Kis,Gre,Sim} to characterize representability in the class of of permutation groups contained in the direct sums of regular groups. This class contains all abelian groups, generalized dicyclic groups (see \cite{God}), regular and semi-regular groups. It may be interesting to note that the two first classes of groups are known for failing to have the so called regular graphical representations (cf. \cite{IW}). In order to obtain our main result we prove first some results that can be of independent interest in the further study of representability. 

\section{Examples}

The following examples are used as special cases in the proofs below. They also provides a good introduction to the techniques we are using in the next section. Below, and in the sequel, for an $n$-tuple $x\in \{0,1\}^n$,
 by $|x|_1$ we denote  the number of 1's occuring in $x$. 
\medskip

\textbf{Example 1.}
Let $$S= \{100000, 010000, 001010, 000101, 111100, 110011 \}$$ and define a boolean function on $\{0,1\}^6$ by $f(x) = 1$ if and only if $x\in S$. Then  $G(f)$ is a permutation group on the set $X =\{1,2,3,4,5,6\}$, i.e. $G(f) \leq S_6$.  The $6$-tuples in $S$ correspond  to subsets of $X$. Subsets of different cardinalities belong necessarily to different orbits in the action of $S_6$ on the subsets, and may be treated as putting  independent conditions on the function $f$. We will use the language of $n$-tuples (and boolean functions), but it is good to keep in mind the corresponding image of subsets (and actions on subsets). In particular, the set $S$ should be viewed as one consisting of three levels (determined by cardinalities of correspondings sets):

$$
\left.\begin{array}{ll}
S = & \{\,100000,010000\\
 &  \;\;\,001010, 000101,\\
&  \;\;\, 111100, 110011\}.
\end{array}
\right.
$$

Since the only $x\in S$ with $|x|_1=2$  are $x=100000$ or $010000$, it follows that all permutations in $S(f)$ preserves 
the orbits $\{1,2\}$ and $\{3,4,5,6\}$. In other words, from the first level above we infer that $S(f) \leq S_2 \oplus S_4$. Thus, every permutation in $S(f)$ is of the form $(\sigma,\tau)$ with $\sigma\in S_2$ and $\tau\in S_4$. The second level implies that no permutation of this form with $\tau = (3,4)$ or $(3,4,5)$ belongs to $G(f)$. In fact, almost all transpositions or  $3$-element cycles for $\tau$ are excluded by this level, and those which are not, are excluded by the third level. Considering a few cases shows also that no permutation with $\tau$ being a $4$-element cycle preserves $S$. Consequently, 
$S(f) \leq S_2 \oplus K_4$, where $K_4$ acts on the set $\{3,4,5,6\}$ and may be written as $K_4 = \{1, (3,4)(5,6), (3,5)(4,6), (3, 6)(4,5)\}$. Now it is easy to check that all permutations in $K_4$ preserves $6$-tuples in $S$, and therefore $S(f) = C_2 \oplus K_4$ (since $S_2=C_2$).
\medskip

\textbf{Example 2.} 
Let $$S= \{100000, 010000, 101010, 010101, 111100, 110011 \}$$ and define a boolean function on $\{0,1\}^6$ by $f(x) = 1$ if and only if $x\in S$. Then  $G(f)$ is a permutation group on the set $X =\{1,2,3,4,5,6\}$, i.e. $G(f) \leq S_6$.  The $6$-tuples in $S$ correspond  to subsets of $X$. Subsets of different cardinalities belong necessarily to different orbits in the action of $S_6$ on the subsets, and may be treated as putting  independent conditions on the function $f$. We will use the language of $n$-tuples (and boolean functions), but it is good to keep in mind the corresponding image of subsets (and actions on subsets). In particular, the set $S$ should be viewed as one consisting of three levels (determined by cardinalities of correspondings sets):

Let
$$
\left.\begin{array}{ll}
S = & \{\,100000,010000\\
 &  \;\;\,101010, 010101,\\
&  \;\;\, 111100, 110011\}.
\end{array}
\right.
$$
There is only a small difference in the second line comparing with the previous example. 
As before we infer that
$S(f) \leq C_2 \oplus K_4$. Yet, now (because of the second line in the array)  not every element of the direct sum is in $S(f)$; for example, $(1,2)(3,5)(4,6)$ and $(1,2)$ are not. 
It is not difficult now to check that $$S(f) = \{ 1, (3,5)(4,6), (1,2)(3,4)(5,6), (1,2)(3,6)(4,5) \}$$.  This group may be viewed as a subdirect sum of $C_2 \oplus_\phi K_4$. In fact, $S(f)=C_2 \oplus_\phi K_4/D$, where $D=\{1, (3,5)(4,6)\}$ and $\phi$ is the the unique isomorphism between $C_2$ and $K_4/D$. This is the unique (up to permutation isomorphism) nontrivial subdirect sum $C_2 \oplus_\phi K_4$.
\medskip

\textbf{Example 3.}
Similarly one can check that for 
$$
\left.\begin{array}{ll}
S = & \{\,001111,\\
 &  \;\;\,101010, 010101, 101100, 010110, 100011, 101001,\\
&  \;\;\, 001010, 000101\}.
\end{array}
\right.
$$
and the corresponding boolean function $f$ on 6 variables,  $$G(f) = C_2 \oplus_\phi C_4/C_2^{(2)}$$ is the unique nontrivial subdirect sum of $C_2$ and $C_4$. The details are left to the reader.  Abstractly, this group happens to be a cyclic group, and therefore is covered by Theorem~\ref{cycCK}. 
\medskip

\textbf{Example 4.}
One may easily check that for $i=3,4,5$, the cyclic group $C_i$ has the same orbits on the $i$-tuples as the dihedral group $D_i \geq C_i$ (note that $D_3=S_3$). It follows that $S(f) \geq C_i$ implies $S(f) \geq D_i$ for any $f$. Consequently, for $i=3,4,5$,  $C_i$ is not representable by any $k$-valued boolean function  (cf. \cite{Kis,CK2}). The same argument works for any  group of the form $C_i\oplus H \leq S_n$: it has the same orbits on the $n$-tuples as $D_i\oplus H$. Therefore,  for $i=3,4,5$,  $C_i\oplus H \notin BGR(k)$  for any $k\geq 2$. 

As an exercise, we leave to the reader to follow Example~3 and construct boolean functions showing that $C_i^{(2)} \in BGR(2)$  for all $i\geq 2$.
\medskip

\section{Results}

We start from a lemma that is a complement of the last example. Recall that by a trivial subdirect sum of $G$ and $H$ we mean just the direct sum $G\oplus H$ (which is the subdirect sum $G/G \oplus_\phi H/H$ with one-element factor groups and the trivial isomorphism). 

\begin{Lemma}
Let $G=C_i/M \oplus_\phi H/N$ be a nontrivial subdirect sum, where $i\in\{3,4,5\}$ and $H$ is regular. If $H\in BGR(2)$, then $G\in BGR(2)$.
\end{Lemma}

\begin{proof} First assume that $M=1$.  It follows that $H/N$ must be isomorphic (abstractly) to $C_i$, $H$ has $di$ elements partitioned into $i$ cosets, where $d$ is the order of $N$. Suppose that $C_i$ acts on $\{1,2,\ldots,i\}$ and $\delta = (1,2,\ldots,i)$ is the generator of $C_i$.

Since $H$ is regular, it follows that it acts on a $di$-element set $X$, say $X=\{i+1,i+2,\ldots, i(d+1)\}$, and there is a partition of $X$ into $d$-element sets $X_1,X_2,\ldots,X_i$ such that $\phi(\tau)(X_k) = X_{k+1}$, where $\phi(\delta)$ is the coset of permutations in $H$, the image of $\delta$ by $\phi$, $k=1,2,\ldots,i$ (here, and in the sequel, we assumme that $i+1$ is to be replaced by $1$). Thus, we have a cyclic action of $H$ on $\{X_1,X_2,\ldots,X_i\}$ (isomorphic to the action of $C_i$).

Let us denote $n=i(d+1)$. We construct a set $S$ consisting of all $n$-tuples corresponding to subsets of $\{1, 2,\ldots,n\}$ of different cardinalities  and then define the boolean function $f=f_S$ corresponding to $S$. 
First we assume that all $n$-tuples of the form $x=u0^{n-i}$ with $u\in\{0,1\}^i$ and $|u|_1=1$ are in $S$ and that these are the only $x\in S$ with $|x|_1=1$. This guarantees that $S(f) \leq S_i \oplus S_{n-i}$. 

In order to make sure that $S(f) \leq S_i \oplus H$ we use the fact that there exists a boolean fuction $h$ such that $H=S(h)$. Let $S_H$ be the set of all $(n-i)$-tuples $v$ such that $h(v)=1$. Then, as the second step, we assume that the only $n$-tuples in $S$ starting from the $i$-tuple $1^i$ are those of the form $x = 1^iv$ obtained as a concatenation of $1^i$ and any $(n-i)$-tuple $v\in S_H$. This guarantees the required property.

Next, we describe those $x\in S$ with $|x|_1=2$ that have exacltly one occurrence of $1$ corresponding to the orbit $\{1,2,\ldots,i\}$. To this end, let us first denote $u_1=10^{i-1}, u_2 = 010^{i-2}, \ldots, u_i=0^{i-1}1$. Then, for each $k=1,2,\ldots,i$, and each $y\in\{0,1\}^{n-i}$ with $|y|_1=1$ corresponding to an element in $X_k$, we assume that $u_ky\in S$. This guarantees a sort of parallel action: for all $r,s \leq i$ and all $(\tau,\sigma)\in S(f)$, if $\tau(r) = s$, then every element in $X_r$ is moved by $\sigma$ into an element in $X_s$.

Finally, similarly, for each $k=1,2,\ldots,i$, and each $z\in\{0,1\}^{n-i}$ with $|z|_1=2$ corresponding to two elements: one in $X_k$ and another in $X_{k+1}$, we assume that $u_kz\in S$ ($i+1$ to be replaced by $1$). This guarantees that the action is isomorphic to that of $C_i$. Indeed, similarly as before: for $(\tau,\sigma)\in S(f)$, if $\tau(r) = s$, then any pair of elements in $X_r$ and $X_{r+1}$ is moved by $\sigma$ into a pair in $X_s$ and $X_{s+1}$. Therefore, if $\tau(r) = s$ ($r,s \leq i$), then $a\in X_r$ implies $\sigma(a) \in X_s$, and consequently, $b\in X_{r+1}$ implies $\sigma(b) \in X_{s+1}$. It follows that both the actions of $\tau$ on  $\{1,2,\ldots,i\}$ and $\sigma$ on $\{X_1, X_2,\ldots,X_i\}$ are cyclic. Since by this construction $I_i \oplus N \leq S(f)$, we infer finally that $S(f)  = C_i \oplus_\phi H/N$.

Under the assumptions of the lemma, the only case with a nontrivial normal divisor $M$ is one with $i=4$ and $M=C_2^{(2)}$. Then  $C_4/M$ and $H/N$ are isomorphic to $C_2$, $H$ has $2d$ elements partitioned into $2$ cosets, where $d$ is the order of $N$. As before, we suppose that $C_4$ acts on $\{1,2,3,4\}$ and $\delta = (1,2,3,4)$ is the generator of $C_4$. Then $M = \{1, (1,3)(2,4)\}$. Since $H$ is regular, it acts on a $2d$-element set $X=\{5,6,\ldots, 2d+4)\}$, and there is a partition of $X$ into two $d$-element sets $X_1,X_2$ such that each permutation in $\phi(\delta M)$ transpose $X_1$ and $X_{2}$, while each permutation in $\phi(M)$ preserves $X_1$ and $X_{2}$.

We follow the previous construction to the point it works properly. 
So, we construct again a set $S$ consisting of all $n$-tuples ($n=2d+4$) corresponding to subsets of $\{1, 2,\ldots,n\}$ and define the boolean function $f=f_S$ corresponding to $S$. 
We assume that all $n$-tuples of the form $x=u0^{n-4}$ with $u\in\{0,1\}^4$ and $|u|_1=1$ are in $S$ and that these are the only $x\in S$ with $|x|_1=1$. This guarantees that $S(f) \leq S_4 \oplus S_{n-4}$.  Also, in order to make sure that $S(f) \leq S_4\oplus H$, we we assume that the only $n$-tuples in $S$ starting from the $4$-tuple $1^4$ are those of the form $x = 1^4v$, where  $v \in \{0,1\}^{(n-4)}$  are those  $(n-4)$-tuples for which $h(v)=1$ for a fixed boolean function with $H=S(h)$. 

Next, we describe those $x\in S$ with $|x|_1=2$ that have exactly one occurrence of $1$ corresponding to the orbit $\{1,2,3,4\}$.  For $u=1000$ or $0010$, and each $y\in\{0,1\}^{n-4}$ with $|y|_1=1$ corresponding to an element in $X_1$, we assume that $uy\in S$. Similarly, for $u=0100$ or $0001$, and each $y\in\{0,1\}^{n-4}$ with $|y|_1=1$ corresponding to an element in $X_2$, we assume that $uy\in S$. 
This guarantees that for all $(\tau,\sigma)\in S(f)$, if $\tau \in M$, then $\sigma$ preserves $X_1$ and $X_2$, and if 
$\tau \in (\delta M$, then $\sigma$ transposes $X_1$ and $X_2$.

Finally, to make sure that the action is as required, we apply a solution slightly different from that applied in the previous construction. For $u=1100$ or $0011$, and each $y\in\{0,1\}^{n-4}$ with $|y|_1=1$ corresponding to an element in $X_1$, we assume that $uy\in S$, and  similarly, for $u=0110$ or $1001$, and each $y\in\{0,1\}^{n-4}$ with $|y|_1=1$ corresponding to an element in $X_2$, we assume that $uy\in S$. 

Now we make use of the fact that $C_4$ is a maximal subgroup of the dihedral group $D_4$, and every group containing properly $C_4$, contains $D_4$. As $C_4$ is assumed to be generated by $\delta=(1,2,3,4)$, $D_4$ contains, in particular, the permutation  $\rho=(12)(34)$. We show that no permutation of the form $(\rho,\sigma)$ belongs to $S(f)$. Indeed, the $n$-tuples $x\in S$ with $|x|_1=2$ shows that if $(\rho,\sigma) \in S(f)$, then $\sigma$ transposes $X_1$ and $X_2$, while 
the $n$-tuples $x\in S$ with $|x|_1=3$ shows that if $(\rho,\sigma) \in S(f)$, then $\sigma$ preserves $X_1$ and $X_2$, a contradiction. Consequently, $S(f) \leq C_4 \oplus N$. Since we took care to make sure that $S(f) \geq C_4/M \oplus_\phi H/N$, it follows now easily that $S(f) = C_4/M \oplus_\phi H/N$.
\end{proof}

We may now generalize the preceding lemma as follows.

\begin{Lemma}\label{ci-reg}
Let $G=C_i/M \oplus_\phi H/N$ be a nontrivial subdirect sum, where $i\in\{3,4,5\}$, and suppose that $H$ is contained in the direct sum of regular groups. If $H\in BGR(2)$, then $G\in BGR(2)$.
\end{Lemma}

\begin{proof} We make a use of the fact that $G$ can be presented as 
$$G = (C_i/M \oplus_{\phi_1} H_1/N_1)  \oplus_{\phi_2} H_2/N_2.$$
Indeed, by assumption $H$ is a subdirect sum of regular factors $H_1,H_2,\ldots,H_r$. Now, $C_I$ has to depend on at least one of these factors, and (considering possible subgroups in each case $i=3,4,5$) it is not difficult to observe that at least one of these dependencies has to be determined by $M$ (i.e. the corresponding constituent has to be of the form $G'=C_i/M \oplus_{\phi_k} H_k/N_k$). To simplify the notation, we assume that $k=1$.

Now, it is enough to modify slightly the proof of the preceding lemma. As before, we construct a set $S$ and the boolean function $f=f_S$ corresponding to $S$. The first two steps are the same, and they guarantee that $S(f) \leq S_i\oplus H$. The remaining steps are also the same, but we apply them to the factor $C_i/M \oplus_{\phi_1} H_1/N_1$ rather than to whole $G$. This guarantees that $S(f) \leq C_i/M \oplus_\psi H/N'$, for some $\psi$ and $N'$. Since the construction guarantees also that $S(f) \geq C_i/M \oplus_\psi H/N$, the result follows.
\end{proof}

We will need a result from \cite{Kis} concerning direct sums, but in a slightly stronger form. So, we now improve the result proved in \cite{Kis}.

\begin{Theorem} \label{dirprod}
If $G\leq S_m$ and $H\leq S_n$ are $k$-representable for some $k\geq 2$
$(n,m\geq 2)$,
then $G\oplus H\leq S_{n+m}$ is $r$-representable for every $r$ satisfying
$r^2\geq k$.
\end{Theorem}

\begin{proof}
Let $G=S(g)$ and $H=S(h)$, where $g$ and $h$ are $k$-valued
functions on $m$ and $n$ boolean variables, respectively. Without
loss of generality we may assume that $m\leq n$. Moreover, we assume that
$k$ different values of $g$ and $h$ are taken from the Cartesian product
$P = \{0,1,2,\ldots,r-1\}\times \{0,1,2,\ldots,r-1\}$. This is possible,
since $r^2\geq k$. Denoting
by $\pi_1$ and $\pi_2$ first and second projection operations on $P$,
respectively, 
we define an $r$-valued function $f$ on $\{0,1\}^{n+m}$ as follows:
$$
f(z) =
\left\{\begin{array}{ll}
\pi_1(g(x)), & \mbox{if } z=x0^n  \mbox{ for some } x\in\{0,1\}^m,x\neq
0^m,1^m, \\
\pi_2(g(x)), & \mbox{if } z=x1^n  \mbox{ for some } x\in\{0,1\}^m,x\neq
0^m,1^m, \\
\pi_1(h(y)), & \mbox{if } z=0^my  \mbox{ for some } y\in\{0,1\}^n,y\neq
0^n,1^n, \\
\pi_2(h(y)), & \mbox{if } z=1^my  \mbox{ for some } y\in\{0,1\}^n,y\neq
0^n,1^n, \mbox{ and } \\ & \hspace{3.5cm}    |y|_1 \neq n-m,\\
\pi_2(h(y)), & \mbox{if } z=xy  \mbox{ for some } x\in\{0,1\}^m,y\in\{0,1\}^n,\mbox{ and } \\ 
                          & \hspace{3.5cm}  |x|_1=1, |y|_1= n-m,\\
1,   & \mbox{if } z = 0^m1^n, \\
0,   & \mbox{otherwise.}
\end{array}\right.
$$
(The definition is written in the form following the proof in \cite{Kis}, but it is helpful to look into it as on one defining the set $S$ of $(m+n)$-tuples getting the value $1$). First, we show in two steps that $S(f)=G\times H$.

At first, if $z_1 = xy$ and $z_2=x^\sigma y$ for some $x\in\{0,1\}^m,
y\in\{0,1\}^n$ and $\sigma\in G$, then obviously $f(z_1)=f(z_2)$ (since
in the two first cases of the definition $g(x)=g(x^\sigma)$, and in the remaining cases applying $\sigma$ to $x$ does not change the case). Similarly, if $z_1 = xy$ and $z_2=xy^\tau$
for some $x\in\{0,1\}^m,
y\in\{0,1\}^n$ and $\tau\in H$, then $f(z_1)=f(z_2)$. It follows, that
$S(f)\supseteq G\times H$.

To prove the converse inclusion we first note that
$S(f) \leq S_n\times S_m$. Indeed, it follows immediately from the fact
that the definition is constructed so that if $|z|_1=n$ then $f(z)= 0$ unless $z=0^m1^n$ in which case $f(z)=1$.

Now, let $\rho = (\sigma,\tau) \notin G\times H$. Then either
$\sigma\notin G$ or $\tau\notin H$. For the latter
there is $y\in\{0,1\}^n$ such that $h(y)\neq h(y^\tau)$ and,
obviously, $y \neq 0^n, 1^n$. Hence,
either $\pi_1(h(y))\neq \pi_1(h(y^\tau))$ or
$\pi_2(h(y))\neq \pi_2(h(y^\tau))$. It follows that
$f(z) \neq f(z^\rho)$ for either $z=0^my$ or $z=1^my$, or $z=xy$ with $|x|_1=1$.
Consequently, $z \notin G\times H$. In the former case, when $\sigma\notin G$, the proof is analogous and even easier (there is no special case with $|y|_1=1$).
Hence, it follows that $S(f)=G\times H$.

Note that in the 5th line of the definition of $f$ the condition ``$|x|_1=1$'' may be replaced by ``$|x|_1=2$'', and the proof remains valid, provided $m>2$. We are going to use the definition with this modification in case of special need.
\end{proof}

Let us note that for $m=1$ the same argument proves that if $H\leq S_n$ is $2$-representable, then $S_1\oplus H$ is also $2$-representable. Moreover, if $H$ is $4$-representable, and $H=S(f)$ for some $4$-valued boolean function $h$ such that $h(y)=0$ for all $y$ with $|y|_1 = n-1$, then $S_1\oplus H$ is $2$-representable. In particular, this applies to $H=K_4$, and will allow us to ignore $S_1$-summands in further considerations.

We generalize also a result formulated in \cite{Sim}. 

\begin{Theorem}
Let $G\leq H \leq S_n$ and $k\geq 2$. If $H = S(f)$ for some $k$-valued boolean function $f$ $(k\geq 2)$ and $H$ has a regular set $S$ such that $f(x)=f(x_S)$ for all $x$ with $|x|_1=|x_S|_1$ (where $x_S$ is the $n$-tuple corresponding to $S$), then $G\in BGR(k)$. 
\end{Theorem}

 The proof is a natural generalization of the proof given in \cite{Sim}.  In the proofs below, rather than applying this theorem, we use directly the method invented by Dalla Volta and Siemons. To this end we will need the following lemma which follows immediately from definitions of \emph{regularity} of permutation groups and sets.

\begin{Lemma}\label{sel}
Let $G\leq H \leq S_n$ and let $H$ be a direct sum of regular groups. If $S=\{t_1,t_2, \ldots, t_r\}$ is a set containing a selector of the orbits of $H$ and contained in a selector of the orbits of $G$, then $S$ is regular in $H$.
\end{Lemma}

Our first application of this approach is given in the following.

\begin{Lemma}\label{no-cyc}
Let $A$ be a permutation group contained in the direct sum $B$ of regular groups. If all the summands of $B$ are different from $S_1$, and  for some $n \geq m \geq 3$ there are $G \leq S_m$ and $H \leq S_n$, such that $B=G\oplus H$, and both $G,H\in BGR(4)$, then $A\in BGR(2)$.
\end{Lemma}

\begin{proof}
Let $G=S(g)$, and $H=S(h)$, where $g$ and $h$ are $4$-valued
functions on $m$ and $n$ boolean variables, respectively. We assume that the values of $g$ and $h$ are taken from the Cartesian product
$P = \{0,1\}\times \{0,1\}$, and define
a boolean (2-valued) function $f$ on $\{0,1\}^{n+m}$ as in the proof of Theorem~\ref{dirprod}. Then, $S(f) = G \oplus H$. 

We modify the definition of $f$ to obtain $S(f)=A$.

Namely, let $S$ be a selector of orbits of $B$. By Lemma~\ref{sel}, $S$ is a regular set in $B$. Let $x_S=uv$ be the $(m+n)$-tuple corresponding to $S$ with $u\in\{0,1\}^m$ and $v\in \{0,1\}^n$. Then, since $B$ has no $S_1$-summands, $m >|u|_1> 0$ and $n > |v|_1 > 0$. In case when $|u|_1=1$ we assume that we consider the alternate form of the definition of $f$ described at the end of the proof Theorem~\ref{dirprod} (that is, one, where the condition ``$|x|_1=1$'' is replaced by ``$|x|_1=2$''). Then, it follows that $f((x_S)^\rho) = 0$ for all $\rho \in G \oplus H$.

Now, in the definition of $f$ we add a new condition that $f(z)=1$ whenever $z=(x_S)^\rho$ with $\rho \in A$. Note that, as there are no $S_1$-summands and $n\geq m\geq n \geq 3$, $|x_S|_1 < n$.  Moreover, since in case when $|u|_1=1$ we apply the alternate form of the definition of $f$, it follows that if $\delta\in S(f)$, then $(x_S)^\delta = (x_S)^\rho$ for some   $\rho \in A$. In other words, the other conditions are independent and therefore, they still guarantee that $S(f) \subseteq B$. Yet, now if 
$\delta \in  S(f)$, then $x_S^\delta = x_S^\rho$, for some $\rho\in A$, and consequently,  $(x_S)^{\delta\rho^{-1}} = x_S$. Since $S$ is regular in $B$, $\delta\rho^{-1} = 1$, and consequently, $\delta=\rho \in A$. Since obviously $A \subseteq S(f)$, it follows that $S(f) = A$.
\end{proof}

We would like to have a similar result for $B=C_2\oplus H$, but in this case there is no room to apply the same argument. Yet, the following will be sufficient for our purposes.

\begin{Lemma}\label{cdwa}
Let $A$ be a permutation group contained in the direct sum $B$ of regular groups such that  $B=C_2\oplus H$ for some $H \leq S_n$, $(n\geq 2)$. If $H\in BGR(2)$ or $H=K_4$, then $A\in BGR(2)$. 
\end{Lemma}

\begin{proof}
To obtain this result it is enough to modify slightly the construction in the previous proof. First, in the definition of $f$ in the proof of Theorem~\ref{dirprod}, we omit the fourth and fifth lines. Then, we further modify the obtained definition, as in the proof above, using the selector $S$. We leave to the reader to check that this  yields the first part of the claim. In the special case when $H=K_4$, the claim follows from Examples~1 and~2.
\end{proof}

Now we are ready to prove our main result.

\begin{Theorem}\label{main}
Let $A$ be a permutation group contained in a direct sum $B$ of regular groups. If $A$ is different from $K_4, C_3,C_4,C_5$ and from any direct sum $C_i \oplus D$ for $i=3,4,5$, then $A\in BGR(2)$. Otherwise, either $A\notin BGR(k)$ for any $k\geq 2$ or $A=K_4\in BGR(3)$.
\end{Theorem}

\begin{proof} By assumption $G$ is a subdirect sum of regular groups. In view ot the remark following Theorem~\ref{dirprod} we may assume that $G$ has no fixpoints. First we consider the case when no summand is of the form $C_3,C_4$, or $C_5$. In this case the result is by induction on the number $r$ of summands. If there is only one summand, then $G$ is transitive and regular, and the claim is true by the result of Grech \cite{Gre}. For induction step we apply Lemmas~\ref{no-cyc} and~\ref{cdwa}. 

It follows that $G$ may be presented as a subdirect sum of groups groups $C_{i_1}, C_{i_2},\ldots,C_{i_k}$ and a group $H$, where $i_j \in \{3,4,5\}$ and either $H\in BGR(2)$ or $H=K_4$, or else $H$ is missing.  If $C_{i_1}$ is idendependent from any other summand then the claim is true by Example~4. Otherwise we show that the number $k$ of $C_{i_j}$ summands may be reduced to $k=0$. Indeed, if $C_{i_1}$ depends on one of the summands $C_{i_j}$ ($j>1$), then by Theorem~\ref{cycCK} and Example~3 the sum $C_{i_1} \oplus_\phi C_{i_j} \in BGR(2)$, and by Lemmas~\ref{no-cyc} and~\ref{cdwa}, it can be included in the summand $H \in BGR(k)$. If $C_{i_1}$ depends on $H$, then by Lemma~\ref{ci-reg}, $H$ may be replaced by  $C_{i_1} \oplus_\phi H\in BGR(2)$.
The proof completes the observation that if $H=K_4$, then for $i=3,4,5$, $C_{i}$ is independent from $H$. 
\end{proof}


\begin{thebibliography}{99}

\bibitem{Baba} L. Babai, 

{\it Automorphism groups, isomorphism, reconstruction,} 

in: Handbook of Combinatorics, 

Elsevier Science B. V. 1995, pp. 1447-1540. 




\bibitem{Bab} L. Babai, 

{\it Finite digraphs with given regular automorphism groups}, 

Period. Math. Hungar. {\bf 11} n. 4 (1980), 257-270.


\bibitem{CK1} P. Clote, E. Kranakis, 

{\it Boolean functions, invariance groups, and parallel complexity,} 

SIAM J Comput. {\bf 20} (1991), 553-590.




\bibitem{CK2} P. Clote, E. Kranakis, Boolean Functions and Computation Models, Springer-Verlag 2002. 

\bibitem{DD} J. D. Dixon, B. Mortimer, Brian, Permutation groups, Berlin, Springer-Verlag 1996.

\bibitem{Sim} F. Dalla Volta, J. Siemons, \emph{Permutation groups defined by unordered relations
}, in: Ischia Group Theory 2008, World Scientific Publishing Co. 2009, pp. 56-67.


\bibitem{Gre} M. Grech, 
{\it Regular symmetric groups of boolean functions}, Discrete Math. 310 (2010), 2877 - 2882.

\bibitem{God} C. D. Godsil, 

{\it $GRRs$ for non solvable groups. Algebraic methods in graph theory}

Colloquia Mathematica Societatis Janos Bolyai (1978)  221-239. 


\bibitem{IW} W. Imrich, M. A. Watkins,
{\it On graphical regular representations of cyclic extensions of groups}
Pacific J. Math {\bf 55} (1974), 461-477.


\bibitem{Kis} A. Kisielewicz, 

{\it Symmetry groups of boolean functions and constructions of permutation groups,} 

J. Algebra {\bf 199}, (1998) 379-403.




\bibitem{Kis2} A. Kisielewicz,
{\it Supergraphs and graphical complexity of permutation groups}, Ars Combinatora 100 (2011), to appear.

\bibitem{KP} M.Ch. Klin, R. Pöschel, K. Rosenbaum,
Angewandte Algebra für Mathematiker und Informatiker, VEB Deutscher Verlag der
Wissenschaften. Berlin 1988. 

\bibitem{Ser}  A. Seress, \emph{Primitive groups with no regular orbits on the set of subsets},  Bull. London Math. Soc. \textbf{29}(6),  (1997)  697-704.



\bibitem{Wie} H. Wielandt,
{\it Permutation groups through invariant relation and invariant functions},
in ``Mathematische Werke/Mathematical works, vol. 1, Group theory'',
Bertram Huppert and Hans Schneider (Eds.), Walter de Gruyter Co.,
Berlin, 1994, pp. 237-266.




\end{thebibliography}
\end{document}